\documentclass[a4paper, 10pt]{amsart}
\usepackage{amssymb,amsmath,amsthm,verbatim,amscd,enumerate,url,hyperref,cleveref}
\usepackage[margin=1.15in]{geometry}

\DeclareMathOperator{\Cay}{Cay}

\title{A characterisation of virtually free groups via minor exclusion}
\author{A. Khukhro}

\parskip\bigskipamount

\theoremstyle{plain}
\newtheorem{Thm}{Theorem}
\newtheorem{Prop}[Thm]{Proposition}
\newtheorem{Corollary}[Thm]{Corollary}
\newtheorem{Lemma}[Thm]{Lemma}

\theoremstyle{definition}
\newtheorem*{thm}{Theorem}

\newtheorem*{question}{Question}
\newtheorem*{acknow}{Acknowledgements}

\begin{document}

\begin{abstract}
We give a new characterisation of virtually free groups using graph minors. 
Namely, we prove that a finitely generated, infinite group is virtually free if and only if for any finite generating set, the corresponding Cayley graph is minor excluded. 
This answers a question of Ostrovskii and Rosenthal. 
The proof relies on showing that a finitely generated group that is minor excluded with respect to every finite generating set is accessible, using a graph-theoretic characterisation of accessibility due to Thomassen and Woess. 
\end{abstract}

\maketitle

\section{Introduction}

A finite graph $\Delta$ is a \emph{minor} of the graph $\Gamma$ if $\Delta$ can be obtained from $\Gamma$ by contracting edges and deleting edges and vertices. A graph $\Gamma$ is said to be \emph{minor excluded} if there exists some finite graph that is not a minor of $\Gamma$. 

Many interesting graph-theoretic properties can be characterised using minor exclusion, or more precisely, exclusion of specific minors. 
The famous theorem of Kuratowski \cite{Kur} states that a graph is planar if and only if it does not have the graphs $K_5$ and $K_{3,3}$ as minors. 
In fact, it is a consequence of the relatively recent Robertson--Seymour Theorem \cite{RS} that any minor-closed property of finite graphs (i.e. a property preserved by passing to a minor) can be characterised via a finite set of forbidden minors -- i.e. for such a property, there is a finite set of graphs such that having this property is equivalent to not possessing any minors in this finite set. 

One can apply graph-theoretic notions to study finitely generated groups, by considering their Cayley graphs. 
Whether or not a given graph appears as a minor can depend on the choice of generating set -- we must therefore choose whether we would like the graph-theoretic property in question to hold for at least one choice of generating set, or for any choice of generating set. 

There already exists a large body of work on the topic of planarity of Cayley graphs. 
The question of which groups admit a planar Cayley graph was first answered for finite groups by Maschke in 1896. 
Since, there has been much study of this property for finitely generated groups, see for example \cite{AC, Dro, DSS, Geo, Ren}. 

The more general property of having excluded minors has not been explored as much in geometric group theory. 
In \cite{AC}, Arzhantseva and Cherix prove that generic (in a certain precise sense) finitely presented groups admit a $K_{2m+1}$ minor, where $m$ is the number of generators of the groups. 
Our main reference for this topic is \cite{OR}, in which Ostrovskii and Rosenthal prove several fundamental results on minor exclusion in Cayley graphs. 
In particular, they show the following.
\begin{itemize}
\item
If $G$ is a finitely generated group with one end, then there is a finite generating set $S$ of $G$ such that $\Cay(G,S)$ is not minor excluded.
\item
There exist groups (e.g. $\mathbb{Z}^2$) that are minor excluded with respect to one choice of generating set, and not minor excluded with respect to another.
\item
Admitting a Cayley graph that is minor excluded is preserved under free products.
\item
Virtually free groups are minor excluded for any choice of generating set. 
\end{itemize}

It is surprisingly often the case in graph theory that an obviously necessary condition also turns out to be sufficient. 
Ostrovskii and Rosenthal ask the following question about the reverse implication of their final result above.
\begin{question}[\cite{OR}, Problem 4.2]
Is an infinite, finitely generated group that is minor excluded with respect to any choice of generating set necessarily virtually free?
\end{question}

Being virtually free is equivalent to many diverse properties, for example, being quasi-isometric to a tree \cite{GdlH}, being finitely presentable with asymptotic dimension equal to $1$ \cite{Gen,JS}, and having context-free word problem \cite{MS}. 
Many of these characterisations use Dunwoody's theorem on the accessibility of finitely presented groups \cite{Dun} (see \cite{Ant} for proofs of several characterisations that do not depend on this result).

We answer the question of Ostrovskii and Rosenthal affirmatively, providing a new graph-theoretic characterisation of virtually free groups. 
We note that this characterisation does not use the accessibility of finitely presented groups. 

\begin{thm}
A finitely generated, infinite group is virtually free if and only if for any finite generating set, the corresponding Cayley graph is minor excluded.
\end{thm}

The ``only if'' direction of the theorem is Theorem 3.7 of \cite{OR}. 
In \Cref{ResultSection}, we first prove that any finitely generated group that is minor excluded with respect to every finite generating set is accessible, using a graph-theoretic characterisation of accessibility due to Thomassen and Woess \cite{TW}.
We then show that if a finitely generated, infinite group is such that all of its Cayley graphs with respect to finite generating sets are minor excluded, then it is virtually free. 
The relevant preliminaries are introduced in Sections \ref{EndSection} and \ref{MinorSection}.

\section{Ends and accessibility}\label{EndSection}

Two one-way infinite paths in a graph $\Gamma$ are said to be equivalent if the tails of the paths remain in the same connected component of $\Gamma \setminus F$ for any finite subset of vertices $F$ of $\Gamma$. 
The set of equivalence classes of one-way infinite paths induced by this equivalence relation is the \emph{set of ends} of $\Gamma$. 

The number of ends of a graph is a quasi-isometry invariant. 
In particular, when the graph in question is the Cayley graph of a finitely generated group $G$, the number of ends is independent of the choice of finite generating set, and we can thus refer to the number of ends of the group $G$.

The following are well-known results linking the number of ends of a group to its algebraic properties.

\begin{Thm}[\cite{Hop}]\label{Hopf}
Let $G$ be a finitely generated group. Then
\begin{itemize}
\item
$G$ has $0$, $1$, $2$, or infinitely-many ends;
\item
$G$ has $0$ ends if and only if it is finite;
\item
$G$ has $2$ ends if and only if it is virtually $\mathbb{Z}$.
\end{itemize}
\end{Thm}

\begin{Thm}[\cite{Sta}]
Let $G$ be a finitely generated group. 
Then $G$ has infinitely many ends if and only if $G$ splits as an amalgamated free product $A*_C B$ or HNN-extension $A*_C$ with $C$ finite and $|A/C|\geq 3$ and $|B/C|\geq 2$.
\end{Thm}

Thus, a group with $\geq 2$ ends can be decomposed into smaller pieces via one of the splittings described above. 
It may be that the groups $A$ and $B$ can themselves be decomposed via an amalgamated free product or HNN extension over a finite subgroup. 
Iterating such (non-trivial) decompositions may not always terminate in a finite number of steps. 
When it does, the group is termed \emph{accessible}.
Otherwise, we say it is \emph{inaccessible}.

Finitely presented groups are known to be accessible by a result of Dunwoody \cite{Dun}, as are groups with a bound on the order of their finite subgroups by a result of Linnell \cite{Lin}.

In \cite{TW}, Thomassen and Woess explore graph-theoretic counterparts to accessibility. 
We now summarise the necessary terminology and results from Section 4 of \cite{TW}.

Given a connected, locally finite graph $\Gamma$, let $\omega$ be an end. 
We say that a pairwise disjoint sequence of vertex subsets $X_0, X_1, \ldots$ is a \emph{defining sequence of vertex sets for $\omega$} if for all $i\geq 0$, 
$$X_{i+1} \subset \Gamma_i \ \ \text{ and } \ \  \Gamma_{i+1} \subset \Gamma_i,$$
where $\Gamma_i$ is the component of $\Gamma \setminus X_i$ which contains $\omega$. 
These conditions imply that a path lies in $\omega$ if and only if it has infinite intersection with $\sqcup_i X_i$.

If there is an integer $m$ such that the $X_i$ can all be chosen to have cardinality $m$, then the end $\omega$ is said to be \emph{thin}, and the minimal such $m$ is called the \emph{size} $s(\omega)$ of $\omega$. 
As shown in the discussion ahead of Proposition 4.1 of \cite{TW}, $s(\omega)$ is the maximum number of pairwise disjoint one-way infinite paths in $\omega$. 

Thomassen and Woess give a graph-theoretic characterisation of accessibility which in the case of finitely generated groups is equivalent to the usual notion of accessibility. 
We summarise the result we will use as follows.

\begin{Thm}[\cite{TW}, Theorem 1.1 and Corollary 8.5]\label{Inaccess}
Let $G$ be a finitely generated group, and let $S$ be a finite generating set of $G$. Then $G$ is inaccessible if and only if for each $m>0$, $\Cay(G,S)$ has a thin end of size $\geq m$. 
\end{Thm}

\section{Graph minors}\label{MinorSection}

Let $\Gamma$ be a graph and let $\Delta$ be a finite graph. 
Write $V(\Delta)$ to mean the vertex set of $\Delta$. 
We say that $\Delta$ is a \emph{minor} of $\Gamma$ if there exist $|V(\Delta)|$ pairwise-disjoint connected finite subsets $V_1, V_2, \ldots V_{|V(\Delta)|}$ of $V(\Gamma)$ that are in bijection with the vertices of $\Delta$, such that there is an edge between some vertex of $V_i$ and some vertex of $V_j$ in $\Gamma$ whenever the corresponding vertices of $\Delta$ are connected by an edge in $\Delta$. 

We say the graph $\Gamma$ is \emph{minor excluded} if there exists some finite graph that is not a minor of $\Gamma$. 
Note that a graph $\Gamma$ is minor excluded if and only if there is some $m\in \mathbb{N}$ such that the complete graph $K_m$ on $m$ vertices is not a minor of $\Gamma$. 
This is because every finite graph is a subgraph of some $K_m$. 

For groups, whether or not the Cayley graph contains a given minor clearly depends on the choice of generating set.
A very simple example is the group $\mathbb{Z}_5$ taken with the generating sets $S:= \{\pm 1\}$ and $T:=\mathbb{Z}_5$: $\Cay(\mathbb{Z}_5, S)$ is clearly planar and does not contain $K_5$ as a minor, while $\Cay(\mathbb{Z}_5,T)$ is precisely $K_5$. 

In \cite{OR}, Ostrovskii and Rosenthal prove that the property of minor exclusion also depends on the generating set, by exhibiting a generating set with respect to which the group $\mathbb{Z}^2$ is not minor excluded i.e. contains all finite graphs as minors; however $\mathbb{Z}^2$ is planar and thus minor excluded with respect to the standard generating set $\{(\pm 1,0),(0,\pm 1)\}$. Note that this gives an example of a $1$-ended group which can be minor excluded, or not, depending on the generating set. 

The concept of ends defined in the previous section is intuitively a relevant one for minor exclusion. 

We can summarise the known relationships between the number of ends of a finitely generated group and minor exclusion as follows.
\begin{itemize}
\item
$0$-ended groups are finite and are thus minor excluded with respect to any generating set.
\item
$1$-ended groups admit a generating set with respect to which they are not minor excluded (\cite{OR}, Theorem 3.11).
\item
$2$-ended groups are virtually free and are thus are minor excluded with respect to any generating set (\cite{OR}, Theorem 3.7).
\end{itemize}

Only the case of groups with infinitely many ends remains. 
In the next section, we show that there do exist groups with infinitely many ends that admit a Cayley graph that is not minor excluded (\Cref{infinity2}). 
This can also be deduced from our main result, \Cref{result}, which characterises virtually free groups as exactly those groups that are infinite and are minor excluded with respect to any generating set.

Finally, we note the following easy lemma, which we will use in the last section.

\begin{Lemma}\label{Subgroup}
Let $G$ be a finitely generated group and let $H$ be a finitely generated subgroup of $G$. 
If $H$ is not minor excluded for some finite generating set $S$, then there exists a finite generating set of $G$ with respect to which $G$ is not minor excluded. 
\end{Lemma}

\begin{proof}
Let $T$ be a finite generating set of $G$. 
The Cayley graph $\Cay(G,S\cup T)$ contains a copy of $\Cay(H,S)$ as a subgraph, an thus any minors of $\Cay(H,S)$ are also minors of $\Cay(G,S\cup T)$. 
\end{proof}

\section{Characterising virtually free groups}\label{ResultSection}

In \cite{OR}, Ostrovskii and Rosenthal show that a virtually free group is minor excluded with respect to any finite generating set. 
We will now show that a finitely generated, infinite group that is minor excluded with respect to any finite generating set is virtually free.  


We will in fact prove that if a finitely generated group is minor excluded with respect to any finite generating set, then it is accessible. 
We then complete the proof of the theorem with the following proposition.

\begin{Prop}\label{Access}
An accessible, infinite, finitely generated group that is minor excluded with respect to any finite generating set is virtually free. 
\end{Prop}

\begin{proof}
Let $G$ be an accessible, infinite, finitely generated group that is minor excluded with respect to any finite generating set. 
Since $G$ is accessible, it can be realised as the fundamental group of a finite graph of groups where each edge group is finite and each vertex group is finitely generated and has $\leq 1$ end. 

By Theorem 3.11 of \cite{OR}, a group with one end admits a finite generating set with respect to which it is not minor excluded. 
Since $G$ is minor excluded for all finite generating sets, by \Cref{Subgroup}, there can be no vertex groups with 1 end. 
Thus, all vertex groups are finite, whence $G$ is virtually free by \cite{KPS}.
\end{proof}

\begin{Thm}\label{Main}
A finitely generated group that is minor excluded with respect to any finite generating set is accessible.
\end{Thm}

\begin{proof} 
We will show that an inaccessible group admits a generating set with respect to which it is not minor excluded. 
Suppose that $G$ is inaccessible, and fix some generating set $S$ of $G$. 
Then by \Cref{Inaccess}, for each integer $m>0$, $\Cay(G,S)$ has a thin end $\omega$ of size $\geq m$, and thus, there are $\geq m$ pairwise disjoint one-way infinite paths in $\omega$. 
We will use these to build a $K_m$ minor in the Cayley graph $\Cay(G,S\cup S^2\cup S^3)$, using a technique from the proof of Lemma 3.12 of \cite{OR}. 
Here, $S^k$ denotes the set $\{s_1 s_2 \cdots s_k: s_i\in S\}$. 
This will show that $\Cay(G,S\cup S^2\cup S^3)$ is not minor excluded. 

Let $R_1, R_2, \ldots R_m$ be vertex subsets in $\Cay(G,S)$ making up $m$ pairwise disjoint one-way infinite paths in $\omega$. 
We will now modify these sets so that they form branch sets of a $K_m$ minor in $\Cay(G,S\cup S^2\cup S^3)$.
We will look at pairs $R_i, R_j$ in turn and show that, possibly after some modification, they can be connected by paths $P_{i,j}$ in $\Cay(G,S\cup S^2\cup S^3)$ which do not intersect each other, nor any of the other (possibly modified) sets $R_k$.

We start with the sets $R_1$ and $R_2$. 
Since $\Cay(G,S)$ is connected, there is a path $P_{1,2}$ linking some vertex of $R_1$ to some vertex of $R_2$.
If $P_{1,2}$ does not intersect any of the other sets $R_3, R_4, \ldots, R_m$, then we leave $R_1, R_2$ and $P_{1,2}$ unchanged. 
If $P_{1,2}$ does intersect some $R_k$, then we can modify the sets to remove this intersection as follows. 

Suppose that the vertices of the one-way infinite path $R_k$ are $\{v_{k,1}, v_{k,2}, \ldots\}$, in the order that they appear as the path goes to infinity. 
If $P_{1,2}$ intersects $R_k$ in just one vertex $v_{k,i}$, then remove the offending vertex from $R_k$. 
The new set, which we will continue to call $R_k$, is still connected and still a one-way infinite path in $\Cay(G,S\cup S^2\cup S^3)$, and remains disjoint from $R_1$ and $R_2$. 
If $P_{1,2}$ intersects $R_k$ in two vertices $v_{k,i}, v_{k,i+1}$ appearing consecutively on the path $R_k$, then remove both from $R_k$. 
Again, the new set $R_k$ is still connected and still a one-way infinite path in $\Cay(G,S\cup S^2\cup S^3)$, and remains disjoint from $R_1$ and $R_2$. 

In case of an intersection that doesn't fall under these two cases, we need to modify both $P_{1,2}$ and $R_k$. Let $v_{k,i}$ be the first vertex of $R_k$ that $P_{1,2}$ intersects, and let $v_{k,j}$ be the last vertex that it intersects, with $j> i+1$. 
Remark first that any number $N\geq 2$ can be decomposed as a sum $N=2r +3s$, with $r,s\in \mathbb{N}$. 
So, since there are $j-i\geq 2$ edges along $R_k$ between $v_{k,i}$ and $v_{k,j}$ in $\Cay(G,S)$, we can find a path $L$ using elements in $S^2 \cup S^3$ between $v_{k,i}$ and $v_{k,j}$, and also a path $L'$ using elements in $S^2\cup S^3$ between $v_{k,i+1}$ and $v_{k,j+1}$, such that the vertices of $L$ and $L'$ lie in $R_k$ and $L$ does not intersect $L'$. 
We then modify the path $P_{1,2}$ by replacing the section between $v_{k,i}$ and $v_{k,j}$ by the path $L$, and we modify $R_k$ by replacing the section between $v_{k,i+1}$ and $v_{k,j+1}$ by the path $L'$. 
The new set $R_k$ now forms a one-way infinite path in $\Cay(G,S\cup S^2\cup S^3)$ that does not intersect $R_1$ and $R_2$, nor the modified path $P_{1,2}$ that joins them. 

The above modifications have removed the intersections of $P_{1,2}$ with $R_k$, without introducing any new intersections with other $R_i$, $i\neq 1,2$. 
We can thus look at each $R_i$, $i\neq 1,2$ in turn and remove any intersections with $P_{1,2}$ via the modifications above. 
We will then have a path $P_{1,2}$ joining $R_1$ and $R_2$ that does not intersect any of the other (possibly modified) $R_i$, $i\neq 1,2$. 

We then proceed inductively to join all pairs of sets $R_i, R_j (i\neq j)$ via paths that do not intersect any of the other $R_k$, nor any of the paths already created. 
Once such a path between one pair has been obtained via the process above, we can then find a ball $B$ about the identity in $\Cay(G,S)$ that is large enough to contain all of the paths $P_{i,j}$ that we have formed thus far, and all of the vertices that were used in modifications of the $R_k$.  
We can then consider $\Cay(G,S)\setminus B$. 
Since the one-way infinite paths $R_i$ correspond to the same end $\omega$, the restrictions of the sets $R_i$ to $\Cay(G,S)\setminus B$ are still in the same connected component of $\Cay(G,S)\setminus B$, whence we can run the above procedure on (the restrictions of) the next pair of one-way infinite paths. 
This next round of modifications of the $R_i$ will not destroy the previous connections, as the new modifications take place outside of $B$.

At the end, we will have paths $P_{i,j}$ for every pair of sets $R_i,R_j$, that do not intersect any of the other $R_k$, nor any of the other paths between them.
Moreover, there is some ball $B'$ of finite radius in $\Cay(G,S\cup S^2\cup S^3)$ containing all of the paths $P_{i,j}$. 
For each path $P_{i,j}$ from $R_i$ to $R_j$, attribute the first half of vertices appearing along the path to $R_i$, and the second half to $R_j$ (making a choice for a vertex in the middle, if it exists). 
For each $i$, take the truncation of the set $R_i$ to $B'$ together with the vertices in the paths $P_{i,j}$, $j\neq i$, attributed to $R_i$. 
These sets now form the branch sets of a $K_m$ minor in $\Cay(G,,S\cup S^2\cup S^3)$.
Since we can do this for any $m$, we see that $\Cay(G,S\cup S^2\cup S^3)$ is not minor excluded. 
\end{proof}

From the proof, we can directly deduce the following incidental corollary, which has an implication in the case of infinitely many ends.

\begin{Corollary}\label{infinity1}
An inaccessible group $G$ is not minor excluded with respect to any generating set of the form $S\cup S^2 \cup S^3$, where $S$ is any generating set of $G$. 
\end{Corollary}

\begin{Corollary}\label{infinity2}
There exist finitely generated groups with infinitely many ends that are not minor excluded with respect to some finite generating set. 
\end{Corollary}

\begin{proof}
By a result of Dunwoody \cite{Du93}, there exist inaccessible groups, which have infinitely many ends. Now apply \Cref{infinity1}. 
\end{proof}

We can now deduce the main theorem, which solves Problem 4.2 of \cite{OR}. 
\begin{Thm}\label{result}
A finitely generated, infinite group is virtually free if and only if it is minor excluded for any finite generating set.
\end{Thm}

\begin{proof}
The ``only if'' direction of the theorem is Theorem 3.7 of \cite{OR}.

For the other implication, \Cref{Main} tells us that such a group must be accessible. 
\Cref{Access} now completes the proof of the theorem.
\end{proof}

As being virtually free is equivalent to being quasi-isometric to a tree (see \cite{GdlH} for a proof), we also immediately obtain the following. 

\begin{Corollary}
For finitely generated groups, the property of being minor excluded with respect to any generating set is a quasi-isometry invariant.
\end{Corollary}

\begin{acknow}
Thanks to Yago Antol\'{i}n, Laura Ciobanu, and Henry Wilton for helpful comments on a first draft of this paper. 
\end{acknow}

\end{document}